\documentclass{amsart}
\newtheorem{theorem}{Theorem}[section]
\newtheorem{lemma}[theorem]{Lemma}

\theoremstyle{definition}
\newtheorem{definition}[theorem]{Definition}
\theoremstyle{remark}

\numberwithin{equation}{section}

\begin{document}
\title{Regularity of Solutions to the Navier-Stokes equations in $\dot{B}_{\infty,\infty}^{-1}$}
    \author{Gregory Seregin}
    \address{Mathematical Institute, University of Oxford, Andrew Wiles Building, Radcliffe Observatory Quarter,
Woodstock Road, Oxford OX2 6GG, United Kingdom}
    \email{seregin@maths.ox.ac.uk}
    
	\author{Daoguo Zhou}
	\address{School of Mathematics and Information Science, Henan Polytechnic University,
		Jiaozuo, Henan 454000, P. R. China}
	\email{daoguozhou@hpu.edu.cn}
\date{}
	\begin{abstract}
		We prove that if $u$ is a suitable weak solution to the three dimensional
		Navier-Stokes equations from the space $L_{\infty}(0,T;\dot{B}_{\infty,\infty}^{-1})$,
		then all scaled energy quantities of $u$ are bounded. As a  consequence,
		 it is shown that  any axially symmetric suitable weak solution $u$,  belonging 
		to  $L_{\infty}(0,T;\dot{B}_{\infty,\infty}^{-1})$, is 
		smooth.
	\end{abstract}
	
	\maketitle
	
	\section{Introduction}
	
	The main aim of this paper is to show  that suitable weak  solutions to the Navier-Stokes equations, whose $\dot{B}^{-1}_{\infty,\infty}$-norm is bounded, have the Type I singularities (or Type I blowups) only. To be more precise in the statement of our results, we need to define certain notions.
	\begin{definition}
		\label{sws} Let $\Omega$ be a domain in $\mathbb R^3$ and let $Q_T:=\Omega\times ]0,T[$. It is said that a pair of functions $v$ and $q$ is a suitable weak solution to the Navier-Stokes equations in $Q_T$ if the following conditions are fulfilled:
		
		(i) $v\in L_\infty(\delta,T;L_{2,loc}(\Omega))\cap L_{2}(\delta,T;W^1_{2,loc}(\Omega)),\quad q\in L_\frac 32(\delta,T;L_{\frac 32,loc}(\Omega))
		$ for any $\delta\in ]0,T]$;

		(ii) $v$ and $q$ satisfy the Navier-Stokes equations
		$$\partial_tv+v\cdot\nabla v-\Delta v=-\nabla q,\qquad {\rm div}\,v=0$$
		in $Q_T$ in the sense of distributions;
		
		(iii) for $Q(z_0,R)\subset \Omega\times ]0,T[$, the local energy inequality
		$$\int\limits_{B(x_0,R)}\varphi|v(x,t)|^2dx+2\int\limits^t_{t_0-R^2}\int\limits_{B(x_0,R)}\varphi|\nabla v|^2dxd\tau\leq 
		$$
		$$\leq \int\limits^t_{t_0-R^2}\int\limits_{B(x_0,R)}\bigg (|v|^2(\partial_t\varphi+\Delta \varphi)+v\cdot\nabla \varphi(|v|^2+2q)\bigg )dxd\tau		$$
	holds for a.a. $t\in ]t_0-R^2,t_0[$ and for all non-negative test functions $\varphi\in C^\infty_0(B(x_0,R)\\
	\times ]t_0-R^2,t_0+R^2[)$.	\end{definition}
	
	Let us introduce the following scaled energy quantities:
	$$ A(z_0,r):=\sup\limits_{t_0-r^2<t<t_0}\frac 1r\int\limits_{B(x_0,r)}|v(x,t)|^2dx,\qquad E(z_0,r):=\frac 1r\int\limits_{Q(z_0,r)}|\nabla v|^2dxdt,
	$$
	$$C(z_0,r):=\frac 1{r^2}\int\limits_{Q(z_0,r)}|v|^3dxdt,\qquad D(z_0,r):=\frac 1{r^2}\int\limits_{Q(z_0,r)}|q|^{\frac 3 2}dxdt,$$
	$$G(z_0,r):=\max\{A(z_0,r),E(z_0,r),C(z_0,r)\},$$$$ g(z_0,r):=\min\{A(z_0,r),E(z_0,r),C(z_0,r)\}.$$
	Here, $Q(z_0,r):=B(x_0,r)\times ]t_0-r^2,t_0[$ and $B(x_0,r)$ is the ball of radius $r$ centred at a point $x_0\in \mathbb R^3$.   
	
	 The important feature of the above quantities is that all of them are invariant with respect to the Navier-Stokes scaling.
	
	Our main result is as follows.
	\begin{theorem}
		\label{maintheorem} Let $\Omega=\mathbb R^3$. Assume that a pair $v$ and $q$ is a suitable weak solution to the Navier-Stokes equations in $Q_T$. Moreover, it is supposed that
		\begin{equation}
			\label{mainassumption}
			v\in {L_\infty(0,T; \dot{B}^{-1}_{\infty,\infty}(\mathbb R^3))}. 
		\end{equation}
	Then, for any $z_0\in \mathbb R^3\times ]0,T]$, we have the estimate
	\begin{equation}
		\label{mainresult}
		\sup\limits_{0<r<r_0}G(z_0,r)\leq c\{[C(z_0,1)+D(z_0,1)]r_0^{\frac 1 2}+\|v\|^2_{L_\infty(0,T; \dot{B}^{-1}_{\infty,\infty}(\mathbb R^3))}+\|v\|^6_{L_\infty(0,T;\dot{ B}^{-1}_{\infty,\infty}(\mathbb R^3))}	\},\end{equation} 	
where $r_0\leq \frac 12\min\{1,t_0\}$ and $c$ is an absolute positive constant.	
			\end{theorem}
			
			Let us recall one of  definitions of the norm in the space $\dot{B}^{-1}_{\infty,\infty}(\mathbb R^3)=\{f \in S': \|f\|_{\dot{B}^{-1}_{\infty,\infty}(\mathbb R^3)}<\infty\}$, which is the following:	
	$$\|f\|_{\dot{B}^{-1}_{\infty,\infty}(\mathbb R^3)}:=\sup\limits_{t>0}t^\frac 1 2\|w\|_{L_\infty(\mathbb R^3)},$$
	where $S'$ is the space of tempered distributions,
	$w$ is the solution to the Cauchy problem for the heat equation with initial datum $f$.
\begin{definition}
	\label{Type1def} Assume that $z_0=(x_0,t_0)$ is a singular point of $v$, i.e., there is no parabolic vicinity of $z_0$ where $v$ is bounded. We call $z_0$ Type I singularity (or Type I blowup) if  there exists a positive number $r_1$  such that 
	$$\sup\limits_{0<r<r_1}g(z_0,r)<\infty.$$
\end{definition}	
	
According to Definition 	\ref{Type1def}, any suitable weak solution, satisfying    assumption (\ref{mainassumption}), has  Type I singularities only. In particular, arguments, used in paper \cite{Seregin2012}, show that axially symmetric suitable weak solutions to  the Navier-Stokes equations have no Type I blowups. This is an improvement of what has been known so far, see papers  \cite{and} and \cite{Seregin2012}, where condition (\ref{mainassumption}) is replaced by stronger one
$$v\in L_\infty(0,T;BMO^{-1}(\mathbb R^3)).$$
Regarding other regularity results on axially symmetric solutions to the Navier-Stokes equations, we refer to papers \cite{BZ2010,CL2002,CSYT2008,CSYT2009,HL2008,JX2003,KNSS2009,La1968,LZ2012,LNZ2016,LZ2017,
Pan2016,Seregin2009,Seregin2007,UY1968,WZ2012,Wei2016}.

Another important consequence is that the smallness of $\|v\|_{L\infty(0,T;\dot{B}^{-1}_{\infty,\infty}(\mathbb R^3))}$ implies regularity, see also \cite{CH2010,HL2017}.

\section{Proof of the Main Result}
	
	In this section,  Theorem \ref{maintheorem} is proved. First, we recall the known multiplicative inequality, see \cite{HMZW2011}.
	
	\begin{lemma}\label{interbesov} For any $u\in\dot{B}_{\infty,\infty}^{-1}(\mathbb{R}^{3})\cap\dot{H}^{1}(\mathbb{R}^{3})$,
		the following is valid:  
\begin{equation}\label{inebesov}
		\|u\|_{L_{4}(\mathbb{R}^{3})}\leq c\|u\|_{\dot{B}_{\infty,\infty}^{-1}(\mathbb{R}^{3})}^{\frac{1}{2}}\|\nabla u\|_{L_2{(\mathbb R^3)}}^{\frac{1}{2}},\end{equation}
where $\dot{H}^1(\mathbb R^3)$ is a homogeneous Sobolev space.
	\end{lemma}
	In fact, a weaker version of (\ref{inebesov}) with $\|u\|_{L^{4,\infty}}$ instead of $ \|u\|_{L_{4}(\mathbb{R}^{3})}$ is needed. Here, $L^{4,\infty}(\mathbb R^3)$ is a weak Lebesgue space. An elementary proof of a weaker inequality is given in \cite{Ledoux2003}.	
	
	The second auxiliary statement is about cutting-off in the space $\dot B^{-1}_{\infty,\infty}(\mathbb R^3)$.
	\begin{lemma}\label{localbesov}
      Let $u\in \dot{B}_{\infty,\infty}^{-1}(\mathbb{R}^{3})$ and $\phi\in C_{0}^{\infty}(\mathbb{R}^{3})$. Then  
      $$\|u\phi\|_{\dot{B}_{\infty,\infty}^{-1}(\mathbb{R}^{3})}\leq c(|{\rm spt}\, \phi|)\|u\|_{\dot{B}_{\infty,\infty}^{-1}(\mathbb{R}^{3})}.$$
      \end{lemma}

We have not found out a proof of Lemma \ref{localbesov} in the literature and presented  it in Appendix. Our proof is elementary and based on typical PDE's arguments.
A scaled version of the previous lemma is as follows.
\begin{lemma}\label{localversion} For any $u\in\dot{B}_{\infty,\infty}^{-1}(\mathbb{R}^{3})\cap H^{1}(B(2))$, the estimate
		\begin{equation}\label{localver}
		\|u\|_{L^{4,\infty}(B)}\leq c\|u\|_{\dot{B}_{\infty,\infty}^{-1}(\mathbb{R}^{3})}^{\frac{1}{2}}\|u\|^{\frac{1}{2}}_{H^{1}(B(2))}.
		\end{equation}
is valid for a universal constant $c$. Moreover, if $u\in\dot{B}_{\infty,\infty}^{-1}(\mathbb{R}^{3})\cap H^{1}(B(x_0,2R))$, then
 \begin{equation}\label{localscaled}
		\|u\|_{L^{4,\infty}(B_{R}(x_{0}))}\leq c\|u\|_{\dot{B}_{\infty,\infty}^{-1}(\mathbb{R}^{3})}^{\frac{1}{2}}\left(\|\nabla u\|_{L_{2}(B_{2R}(x_{0}))}+\frac{1}{R}\|u\|_{L_{2}(B_{2R}(x_{0}))}\right)^{\frac{1}{2}}
		\end{equation}
		with a universal constant $c$.
	\end{lemma}
	Here, we use notation for the ball centred at the origin $B(R)=B(0,R)$ and $B=B(1)$.
	\begin{proof}
It follows from Lemma \ref{interbesov} that for all $\phi\in C_{0}^{\infty}(\mathbb{R}^{3})$, 
		$$
		\|u\phi\|_{L^{4,\infty}(\mathbb{R}^{3})}\leq c\|u\phi\|_{\dot{B}_{\infty,\infty}^{-1}(\mathbb{R}^{3})}^{\frac{1}{2}}\|u\phi\|_{\dot{H}^{1}(\mathbb{R}^{3})}^{\frac{1}{2}}.
		$$
Taking a cut-off function $\phi$ such that $\phi=1$ in
		$B$, $\phi=0$ out of $B(2)$, and $0\leq\phi\leq1$ for $1\leq|x|\leq2$,
		we get inequality \eqref{localver} from Lemma \ref{localbesov}.
		
		To prove inequality \eqref{localscaled}, one  can use scaling and shift  $x=x_0+Ry$, $x\in B(x_0,2R)$, $y\in B(2)$   in \eqref{localver}.
	\end{proof}

In order to prove the main result, we need the following auxiliary inequalities for  $C(z_{0},r)$.
	\begin{lemma}
		For  any $0<r\leq R<\infty$, we have
		\begin{equation}\label{C}
		C(z_{0},r)\leq c\|u\|_{L_{\infty}(0,T;\dot{B}_{\infty,\infty}^{-1}(\mathbb{R}^{3}))}^{\frac{3}{2}}\left(A^{\frac{3}{4}}(z_{0},2r)+E^{\frac{3}{4}}(z_{0},2r)\right),
		\end{equation}
		and
		\begin{equation}\label{C2}
		C(z_{0},r)\leq c\|u\|_{L_{\infty}(0,T;\dot{B}_{\infty,\infty}^{-1}(\mathbb{R}^{3}))}^{\frac{3}{2}}
		\left(\frac{R}{r}\right)^{\frac{3} {4}}\left(A^{\frac{3}{4}}(z_{0},R)+E^{\frac{3}{4}}(z_{0},R)\right).
		\end{equation}
	\end{lemma}
	\begin{proof} Obviously, \eqref{C2} easily follows from \eqref{C}. So, we need to prove the first inequality only. By the H\"older inequality, we have 
	$$\|u(\cdot,t)\|_{L_3(B(x_0,r)}\leq cr^\frac 14\|u(\cdot,t)\|_{L^{4,\infty}(B(x_0,r)}
	$$
	and thus, by \eqref{localscaled},
		$$
			C(z_{0},r) =\frac{1}{r^{2}}\int_{t_{0}-r^{2}}^{t_{0}}\|u(\cdot,t)\|_{L_{3}(B(x_{0},r))}^{3}dt
			$$$$\leq c\frac{1}{r^{\frac{3}{4}}}\|u\|_{L_\infty(0,T;\dot{B}_{\infty,\infty}^{-1}(\mathbb{R}^{3}))}^{\frac{3}{2}}\Big(\int_{t_{0}-(2r)^{2}}^{t_{0}}\|\nabla u(\cdot,t)\|_{L_{2}(B(x_{0},2r))}^{2}+$$$$+\frac{1}{r^{2}}\|u(\cdot,t)\|_{L_{2}(B(x_{0},2r))}^{2}dt\Big)^{\frac{3}{4}}
			$$$$\leq c\frac{1}{r^{\frac{3}{4}}}\|u\|_{L_\infty(0,T;\dot{B}_{\infty,\infty}^{-1}(\mathbb{R}^{3}))}^{\frac{3}{2}}\Big(\int_{t_{0}-(2r)^{2}}^{t_{0}}\|\nabla u(\cdot,t)\|_{L_{2}(B(x_{0},2r))}^{2}dt+
			$$$$+\sup\limits_{-(2r)^{2}+t_0\leq t<t_0}\|u(\cdot,t)\|_{L_{2}(B(x_{0},2r))}^{2}\Big)^{\frac{3}{4}}.
		$$
		This completes the proof of inequality \eqref{C}.
		
	\end{proof}
	
	Now we are going to jusify our main result.
	\begin{proof}[Proof of Theorem \ref{maintheorem}] From the local
		energy inequality, it follows that, for any $0<r<\infty$,
		\begin{equation}\label{en1}
		A(z_{0},r)+E(z_{0},r)\leq c\left(C^{\frac{2}{3}}(z_{0},2r)+C(z_{0},2r)+D(z_{0},2r)\right).
		\end{equation}
		For the pressure $q$, we have the decay estimate
		\begin{equation}
		D(z_{0},r)\leq c\left(\frac{r}{R}D(z_{0},R)+\left(\frac{R}{r}\right)^{2}C(z_{0},R)\right).\label{pressureD}
		\end{equation}
		which is valid for any $0<r<R<\infty$.		
		
		Assume that $0<r\leq\frac \rho 4<\rho\leq1$. Combining \eqref{pressureD} and  \eqref{en1}, we find
		$$
			A(z_{0},r)+E(z_{0},r)+D(z_{0},r) \leq  
		$$$$ \leq c\left(C^{\frac{2}{3}}(z_{0},2r)+C(z_{0},2r)+\left(\frac{\rho}{r}\right)^{2}C(z_{0},\frac{\rho}{2})+\frac{r}{\rho}D(z_{0},\frac{\rho}{2})\right).
		$$
		Now, let us  estimate each term on the right hand side of the last inequality. From  \eqref{C}, \eqref{C2}, and Young's inequality with
		an arbitrary positive constant $\delta$,  we can derive
		\[C(z_0,2r)\leq c\delta\left(A(z_0,\rho)+E(z_0,\rho)\right)+c\delta^{-3}\|u\|_{L^{\infty}(0,T;\dot{B}_{\infty,\infty}^{-1}(\mathbb{R}^3))}^6\left(\frac{\rho}{r}\right)^3.\]
		Similarly, 
		\[C^{\frac{2}{3}}(z_0,2r)\leq c\delta\left(A(z_0,\rho)+E(z_0,\rho)\right)+c\delta^{-1}\|u\|_{L^{\infty}(0,T;\dot{B}_{\infty,\infty}^{-1}(\mathbb{R}^3))}^2\left(\frac{\rho}{r}\right),\]
		and 
		\[(\frac{\rho}{r})^{2}C(z_0,\frac{\rho}{2})\leq c \delta\left(A(z_0,\rho)+E(z_0,\rho)\right)+c\delta^{-3}\|u\|_{L^{\infty}(0,T;\dot{B}_{\infty,\infty}^{-1}(\mathbb{R}^3))}^6\left(\frac{\rho}{r}\right)^8.\]
		Denote $\mathcal{E}(r)=A(z_0,r)+E(z_0,r)+D(z_0,r)$. By a simple inequality $D(z_0,\rho/2)\leq cD(z_0,\rho)$, 
	$$\mathcal{E}(r)\leq c (\delta+\frac r \rho) \mathcal{E}(\rho)+c\Big\{\|u\|_{L_{\infty}(0,T;\dot{B}_{\infty,\infty}^{-1}(\mathbb{R}^3))}^2\Big(\frac{\rho}{r}\Big)\delta^{-1}+$$
	$$+\|u\|_{L_{\infty}(0,T;\dot{B}_{\infty,\infty}^{-1}(\mathbb{R}^3))}^6\Big[\Big(\frac{\rho}{r}\Big)^3+\Big(\frac{\rho}{r}\Big)^8\Big]\delta^{-3}\Big\}.
	$$

		Letting  $r=\theta \rho$ and $\delta=\theta$ and  picking up $\theta$ such that $2c\theta^{1/2}\leq 1$, we find
		$$\mathcal{E}(\theta \rho)\leq \theta^{1/2}\mathcal{E}(\rho)+c\left\{\|u\|_{L_{\infty}(0,T;\dot{B}_{\infty,\infty}^{-1}(\mathbb{R}^3))}^2\theta^{-2} +\|u\|_{L_{\infty}(0,T;\dot{B}_{\infty,\infty}^{-1}(\mathbb{R}^3))}^{6}\theta^{-11}\right\}.$$

		Standard iteration gives us that for  $0<r\leq \frac{1}{2}$,
		$$
		\mathcal{E}(r)\leq c\left(r^{\frac{1}{2}}\mathcal{E}(1)+\|u\|_{L_{\infty}(0,T;\dot{B}_{\infty,\infty}^{-1}(\mathbb{R}^3))}^2+\|u\|_{L_{\infty}(0,T;\dot{B}_{\infty,\infty}^{-1}(\mathbb{R}^3))}^6\right).
		$$
		
	Taking into account  \eqref{C}, we get in addition that
		
		$$
		C(z_{0},r)\leq c\left(r^{\frac{1}{2}}\mathcal{E}(1)+\|u\|_{L_{\infty}(0,T;\dot{B}_{\infty,\infty}^{-1}(\mathbb{R}^3))}^6\right).
		$$
		This completes the proof of Theorem \ref{maintheorem}.
	\end{proof} 
   \appendix 
   
   \section{Proof of Lemma \ref{localbesov}.} We let $w(\cdot,t)=S(t)f(\cdot)$ and 
   $w_\varphi(\cdot,t) =S(t)\varphi f(\cdot)$, where $S(t)$ is a solution operator of the  Cauchy problem for the heat equation with  the initial data $f$ and $\varphi f$, respectevely.
   Then $u:=w\varphi-w_\varphi$ satisfies the equation
   $$\partial_tu-\Delta u=-2{\rm div}\,(\nabla \varphi w)+w\Delta\varphi$$
   and the initial condition $u(\cdot,0)=0$. A unique solution to the problem is as follows:
  $$ u(x,t)=I+J,$$
where 
$$I=-\int\limits^t_0\int\limits_{\mathbb R^3}\Gamma(x-y,t-\tau)
(2{\rm div}\,(\nabla \varphi w))(y,\tau)dy d\tau,$$
$$J=\int\limits^t_0\int\limits_{\mathbb R^3}\Gamma(x-y,t-\tau)
(w\Delta \varphi)(y,\tau)dy d\tau,$$
and $\Gamma$ is the heat kernel.

Let us evaluate $I$. We abbreviate 
$$A:=\|f\|_{\dot B^{-1}_{\infty,\infty}(\mathbb R^3)}=\sup\limits_{t>0}\sqrt t\|w(\cdot,t)\|_{L_\infty(\mathbb R^3)}$$
and $\Omega={\rm spt}\,\varphi$. Then
we have 
$$\sqrt t|I|\leq 2A\sqrt t\int\limits^t_0\frac 1{\sqrt\tau} \int\limits_\Omega
\frac 1{(4\pi(t-\tau))^\frac 32}\exp\Big\{-\frac {|x-y|^2}{4(t-\tau)}\Big\}\frac {|x-y|}{t-\tau}dy d\tau\leq 
$$
$$\leq cA\int\limits^t_0\sqrt{\frac t\tau}\frac 1{(t-\tau)^2}\int\limits_\Omega\exp\Big\{-\frac {|x-y|^2}{4(t-\tau)}\Big\}\frac {|x-y|}{\sqrt{t-\tau}}dy d\tau=
$$
$$=cA\Big[\int\limits^{\frac t2}_0...+\int\limits_{\frac t2}^t...\Big]=cA(I_1+I_2).$$
Regarding $I_1$,  consider first the case $0<t<1$. By the standard change of variables, we have 
$$I_1\leq cAC_0\int\limits^{\frac t2}_0\sqrt{\frac t\tau}
\frac 1{\sqrt{t-\tau}}d\tau$$
with 
$$C_0=\int\limits_{\mathbb R^3}\exp\{-|u|^2\}|u|du.$$
And thus $I_1\leq cA$. In the second case $t\geq 1$,
$$I_1\leq c|\Omega|\int\limits^{\frac t2}_0\sqrt{\frac t\tau}
\frac 1{(t-\tau)^2}d\tau\leq c|\Omega|\frac 1t\leq c|\Omega|.$$

Now, let us evaluate $I_2$. Obviously,
$$I_2\leq c\int\limits^t_{\frac t2}\frac 1{(t-\tau)^2}\int\limits_\Omega\exp\Big\{-\frac {|x-y|^2}{4(t-\tau)}\Big\}\frac {|x-y|}{\sqrt{t-\tau}}dy d\tau.$$
Make change of variables $\vartheta=t-\tau$, then
$$I_2\leq c\int\limits^{\infty}_0 \frac 1{\vartheta^2}\int\limits_\Omega
\exp\Big\{-\frac {|x-y|^2}{4\vartheta}\Big\}\frac {|x-y|}{\sqrt{\vartheta}}dy d\vartheta=$$
$$=c\int\limits^1_0...+c\int\limits^\infty_1...=J_1+J_2.$$
For $J_1$, we have
$$J_1\leq cC_0\int\limits^1_0\frac 1{\sqrt\vartheta}d\vartheta\leq cC_0.$$ Finally, $J_2$ is bounded as follows:
$$J_2\leq c\int\limits^\infty_1\frac 1{\vartheta^2}d\vartheta\leq c|\Omega|.$$

The quantity $J$ is estimated in the same way. Lemma \ref{localbesov} is proved.

\textbf{Acknowledgement} The first author is supported by the grant RFBR 17-01-00099-a.
The second author thanks Professors Cheng He and
Zhifei Zhang for helpful discussions.

\end{document}